\documentclass[a4paper,12pt]{article}
\usepackage{a4wide}
\usepackage{amsmath}
\usepackage{amssymb}
\usepackage{amsthm}
\usepackage{latexsym}
\usepackage{graphicx}
\usepackage[english]{babel}
\usepackage{makeidx}
\usepackage{mathtools}
\usepackage{authblk}

\def\Ch{\operatorname{Ch}}
\def\Char{\operatorname{Char}}

\def\Gal{\operatorname{Gal}}

\def\Hilb{\operatorname{Hilb}}

\def\Irr{\operatorname{Irr}}

\def\Ar{\operatorname{Ar}}
\def\Hol{\operatorname{Hol}}

\def\Mon{\operatorname{Mon}}

\def\ord{\operatorname{ord}}
\def\Cons{\operatorname{Cons}}

\newtheorem{dfn} [subsection]{Definition}

\newtheorem{exm} [subsection]{Example}

\newtheorem{prop}[subsection]{Proposition}

\newtheorem{teor}[subsection]{Theorem}

\numberwithin{equation}{section}

\begin{document}

\title{Weak almost monomial groups and Artin's conjecture}
\date{}







\author{Mircea Cimpoea\c s$^1$}
\maketitle

\begin{abstract}
We introduce a new class of finite groups, called weak almost monomial, which generalize two different notions of "almost
monomial" groups, and we prove it is closed under taking factor groups and direct products.

Let $K/\mathbb Q$ be a finite Galois extension with a weak almost monomial Galois group $G$ and $s_0\in \mathbb C\setminus \{1\}$.
We prove that Artin conjecture's is true at $s_0$ if and only if the monoid of holomorphic Artin $L$-functions at $s_0$ is factorial. 
Also, we show that if $s_0$ is a simple zero for some Artin $L$-function associated to an irreducible character of $G$ and it is not 
a zero for any other $L$-function associated to an irreducible character, then Artin conjecture's is true at $s_0$.



\end{abstract}

\footnotetext[1]{ \emph{Mircea Cimpoea\c s}, National University of Science and Technology Politehnica Bucharest, Faculty of
Applied Sciences, 
Bucharest, 060042, Romania and Simion Stoilow Institute of Mathematics, Research unit 5, P.O.Box 1-764,
Bucharest 014700, Romania, E-mail: mircea.cimpoeas@upb.ro,\;mircea.cimpoeas@imar.ro }

\section{Introduction}\label{sec1}

A finite group $G$ is called \emph{monomial}, if every complex irreducible character $\chi$ of $G$ is monomial, that is it
induced by a linear character $\lambda$ of a subgroup $H$ of $G$ ($\chi=\lambda^G$). 
Let $K/\mathbb Q$ be a Galois extension. For a character of the Galois group $G=\Gal(K/\mathbb Q)$, Artin \cite{artin2} associated a $L$-function, denoted
by $L(s,\chi)$, and he conjectured that $L(s,\chi)$ is holomorphic on $\mathbb C\setminus\{1\}$. Brauer \cite{brauer} proved that $L(s,\chi)$ is meromorphic on $\mathbb C\setminus\{1\}$.

Artin established his conjecture for the monomial characters. Since $$L(s,\phi+\psi) = L(s,\phi)L(s,\psi),$$ for any characters $\phi$ and $\psi$, it follows that Artin's conjecture holds for monomial groups.
 
However, the class of monomial groups is quite restrictive, hence certain generalizations where considered in the literature, related with the study of the holomorphy of Artin $L$-functions. For instance, there are the following 
two different notions of "almost monomial" groups:


A finite group $G$ is called {\em almost monomial}, in the sense of Nicolae \cite{monat}, (shortly, a NAM-group) if for every distinct complex irreducible characters $\chi$ and $\chi'$ of  $G$ there exist a subgroup $H$ of $G$ and a linear character $\lambda$ of $H$ such that the induced character $\lambda^G$ contains $\chi$ and does not contain  $\chi'$; see Definition \ref{nam}. 

A finite group $G$ is called {\em almost monomial}, in the sense of Booker, (or a BAM-group) 
if for each irreducible character $\chi$ of $G$,
if $\chi=\psi+\phi$ for virtual characters $\psi$ and $\phi$ such that
$\langle \psi,\sigma \rangle \geq 0$ and $\langle \phi,\sigma \rangle \geq 0$,
for all monomial characters $\sigma$, then either $\psi=0$ or $\phi=0$; see Definition \ref{bam}. 

In our paper, we introduce a further generalization: We say that a finite group $G$ is {\em weak almost monomial} (or a WAM-group) 
if if for every distinct complex irreducible characters 
$\chi$ and $\chi'$ of  $G$ there exist a subgroup $H$ of $G$ and a linear character $\lambda$ of $H$ such that
$\langle \chi,\lambda^G \rangle > \langle \chi',\lambda^G \rangle$; see Definition \ref{wam}. 

It is clear that if $G$ is a almost monomial, in the sense of Nicolae, it is also weak almost monomial. Also, we prove that
if $G$ is almost monomial, in the sense of Booker, then $G$ is weak almost monomial; see Proposition \ref{bnwam}. Hence,
our notion of weak almost monomial groups generalizes both notions of almost monomial groups. We provide several examples
to clarify these concepts; see Examples \ref{ex1}, \ref{ex2} and \ref{ex3}. 


In Theorem \ref{t1} and Theorem \ref{t2} we prove that the class of weak almost monomial groups is closed under
taking factor groups and direct products, similarly to the class of NAM-groups.
Note that, if $G_1$ and $G_2$ are BAM-groups, it is not known, in general, that 
$G_1\times G_2$ is also a BAM-group. Such result holds under the stronger assumption that
$G_1$ or $G_2$ is monomial; see \cite[Proposition 2.2]{book}.

This is the reason why, we believe that the class of WAM-groups is strictly larger than the class of BAM-group.
A possible example of a WAM-group which is not a BAM-group could be of the form $G_1\times G_2$ with $G_1$ and $G_2$
BAM-groups which are not monomial. Since any BAM-group is a WAM-group and the direct product of two WAM-groups is
a WAM-groups, for sure, $G_1\times G_2$ is weak almost monomial. Unfortunately, we were not able to find such an
example.


Let $s_0\in \mathbb C \setminus \{1\}$ be fixed. Let $\Ar$ be the monoid of Artin $L$-functions and let
$\Hol(s_0)$ be the submonoid of $\Ar$, consisting in L-functions which are holomorphic at $s_0$. 
Nicolae \cite{numb} proved that $\Hol(s_0)$ is an affine monoid. If $G=\Gal(K/\mathbb Q)$ is weak almost monomial then
we prove that Artin's conjecture is true at $s_0$, that is $\Hol(s_0)=\Ar$, if and only if the monoid $\Hol(s_0)$
is factorial; see Theorem \ref{t3}. This result improves \cite[Theorem 1]{monat}.

Also, in Theorem \ref{t4}, we show that if $G$ is weak almost monomial and $s_0$ is a simple zero for 
some Artin $L$-function associated to an irreducible character of $G$ and it is not a zero for any other $L$-function associated to an irreducible character, then Artin conjecture's is true at $s_0$.

In the last section, we present our functions in GAP \cite{gap} which determines that a group $G$ is weak almost monomial and
compute the Hilbert basis of the monoid generated by the monomial characters of $G$ and we give some examples of classical
groups which are or are not weak almost monomial.

\section{Weak almost monomial groups}

First, we fix some notations, following \cite{cimrad}. 
Let $G$ be a finite group with $\Irr(G)=\{\chi_1,\ldots,\chi_r\}$, the set of irreducible (complex) characters.
It is well known that any character $\chi$ of $G$ can be written uniquely as $$\chi=a_1\chi_1+\cdots+a_r\chi_r,$$
where $a_i$'s are nonnegative integers and at least on of them is positive. 

A \emph{virtual character} $\chi$ of $G$ is a linear
combination $\chi:=a_1\chi_1+\cdots+a_r\chi_r$, with $a_i\in\mathbb Z$. The set $\Char(G)$ of virtual characters of $G$ has
a structure of a ring and it is called the \emph{character ring} of $G$. In particular, $(\Char(G),+)$ is an abelian group,
isomorphic to $(\mathbb Z^r,+)$. We denote $\Ch_+(G)$, the set of characters of $G$. Note that 
$\Ch(G):=\Ch_+(G)\cup\{0\}$ is a submonoid of $(\Char(G),+)$, isomorphic to $(\mathbb N^r,+)$.

A character $\chi$ is \emph{monomial}, if there exist a subgroup $H\leqslant G$ and a linear character $\lambda$ of $H$ such that $\lambda^G=\chi$.
We denote $\Mon(G)$, the set of monomial characters of $G$. We let $M(G)$ be the submonoid of $\Ch(G)$, generated by $\Mon(G)$, that is:
$$M(G)=\{\chi\in \Ch(G)\;:\;\chi=\psi_1+\cdots+\psi_m
,\;\psi_i\in\Mon(G),\text{ for all }1\leq i\leq m\}\cup\{0\}.$$
Note that $M(G)$ is a finitely generated submonoid of $\Ch(G)$, hence it is affine. Moreover, $M(G)$ is positive, as $0$ is the only unit of it.

Let $\Hilb(M(G))\subset \Mon(G)$ be the Hilbert basis of $M(G)$, that is the minimal system of generators.
For convenience, we assume that $\Hilb(M(G))=\{\sigma_1,\ldots,\sigma_p\}$, where $\sigma_t=a_{t1}\chi_1+\cdots+a_{tr}\chi_r,$
for some nonnegative integers $a_{tj}$. Since $M(G)$ is isomorphic to a submonoid of $\mathbb N^r$, by abuse of notation, we denote
$$\sigma_t=(\sigma_t(1),\ldots,\sigma_t(r))=(a_{t1},\ldots,a_{tr})\in\mathbb N^r,\;1\leq t\leq p.$$
According to Brauer's induction Theorem, any irreducible character $\chi$ of $G$, can be written as a linear combination of monomial characters,
with integer coefficients. Therefore, $(\Char(G),+)$ is generated, as a group, by $\Mon(G)$. In particular, $\Hilb(M(G))$ generates $(\Char(G),+)$ as 
a group, $M(G)$ is a monoid of the rank $r$ and therefore $p\geq r$.

Note that the group $G$ is monomial if and only if $\Hilb(M(G))=\Irr(G)$.

Booker \cite{book} introduced the following generalization of monomial groups:

\begin{dfn}\label{bam}
A finite group $G$ is called {\em almost monomial} in Booker's sense (or a BAM-group) if for each irreducible character $\chi_i$ of $G$,
if $\chi_i=\psi+\phi$ for virtual characters $\psi$ and $\phi$ such that
$$\langle \psi,\sigma \rangle \geq 0\text{ and }\langle \phi,\sigma \rangle \geq 0,$$
for all monomial characters $\sigma$, then either $\psi=0$ of $\phi=0$. 
\end{dfn}

For a character $\psi$ of $G$, we denote $\Cons(\psi)$ the set of constituents of $\psi$.
Nicolae \cite{monat} introduced an unrelated notion of almost monomial groups:

\begin{dfn}\label{nam}
A finite group $G$ is called {\em almost monomial} in Nicolae's sense (or a NAM-group) if for every two distinct characters $\chi_i \neq \chi_j \in \Irr(G)$, 
there exists a monomial character $\sigma$ such that $\chi_i\in \Cons(\sigma)$ and $\chi_j\notin \Cons(\sigma)$.
\end{dfn}

We introduce the following definition which, as we will see, generalizes both Definition \ref{bam} and Definition \ref{nam}.

\begin{dfn}\label{wam}
A finite group $G$ is called {\em weak almost monomial} (or a WAM-group) if for every two distinct characters $\chi_i \neq \chi_{j} \in \Irr(G)$, there exist
a monomial character $\sigma$ such that $\langle \chi_{i},\sigma \rangle > \langle \chi_{j},\sigma \rangle$.
\end{dfn}

\begin{prop}\label{bnwam}
Let $G$ be a finite group. Then:
\begin{enumerate}
\item[(1)] If $G$ is a BAM-group then $G$ is a WAM-group.
\item[(2)] If $G$ is a NAM-group then $G$ is a WAM-group.
\end{enumerate}
\end{prop}

\begin{proof}
(1) We assume, by contradiction, that $G$ is not weak almost monomial. It follows that there are some indices $i\neq j$ such that
    $$\langle \chi_i,\sigma \rangle \leq \langle \chi_j,\sigma \rangle,$$
		for any monomial character $\sigma$ of $G$. We can write 
		$$\chi_j = \psi + \phi,\text{ where }\psi=\chi_j-\chi_i\text{ and }\phi=\chi_i.$$
		It follows that $\langle \psi,\sigma \rangle \geq 0\text{ and }\langle \phi,\sigma \rangle \geq 0$, for any
		monomial character $\sigma$, which contradicts the definition of a BAM-group.

(2) It follows immediately from the fact that $\chi_i\in \Cons(\sigma)$ is equivalent to $\langle \chi_i,\sigma \rangle > 0$
    and $\chi_j\notin \Cons(\sigma)$ is equivalent to $\langle \chi_j,\sigma \rangle = 0$.
\end{proof}

Note that it is enough to check Definition \ref{bam}, Definition \ref{nam} and Definition \ref{wam} only for the monomial
characters belonging to the Hilbert basis of $M(G)$.

\begin{prop}\label{gogu}
Let $G$ be a finite group with $\Hilb(M(G))=\{\sigma_1,\ldots,\sigma_p\}$. Assume that 
$\sigma_t\in \{0,1\}^r$ for all $1\leq i\leq p$ (we regard characters as vectors in $\mathbb N^r$).
Then $G$ is a WAM-group if and only if $G$ is a NAM-group.
\end{prop}

\begin{proof}
From the hypothesis, the condition $\langle \chi_{i},\sigma_t \rangle > \langle \chi_{j},\sigma_t \rangle$ implies
$\langle \chi_{i},\sigma_t \rangle=1$ and $\langle \chi_{j},\sigma_t \rangle=0$, that is $\chi_i$ is a constituent of $\sigma_t$,
while $\chi_{j}$ is not.
\end{proof}

\begin{exm}\rm(See also \cite[Example 2.7]{cimrad})\label{ex1}
Let $G=SL_2(\mathbb F_3)$. It is well known that $\Irr(G)=\{\chi_1,\ldots,\chi_7\}$, where $\chi_1,\chi_2,\chi_3$ are linear,
$\chi_4, \chi_5, \chi_6$ have degree $2$ and $\chi_7$ has degree $3$. The characters $\chi_1,\chi_2,\chi_3$ and $\chi_7$ are monomial,
while $\chi_4$, $\chi_5$ and $\chi_6$ are not. Using GAP \cite{gap}, we deduce that $\Hilb(M(G))=\{\sigma_1,\ldots, \sigma_8\}$ where
$$\sigma_i=\chi_i,\;i=1,2,3,\;\sigma_4=\chi_7,\;\sigma_5=\chi_4+\chi_5,\;\sigma_6=\chi_4+\chi_6,\;
\sigma_7=\chi_5+\chi_6,\;\sigma_8=\chi_4+\chi_5+\chi_6.$$
It is easy to check that $G$ is a NAM-group. Moreover, according to \cite[Proposition 2.3]{book}, $G$ is a BAM-group.
\end{exm}

\begin{exm}\rm(See also \cite[Example 2.7]{cimrad})\label{ex2}
Let $G=GL_2(\mathbb F_3)$. It is well known that $\Irr(G)=\{\chi_1,\ldots,\chi_8\}$, where $\chi_1,\chi_2,\chi_3,\chi_6,\chi_7$ and
$\chi_8$ are monomial, while $\chi_4$ and $\chi_5$ are not. Using GAP \cite{gap}, we deduce that 
$\Hilb(M(G))=\{\sigma_1,\ldots, \sigma_9\}$ where
$$\sigma_i=\chi_i,\;i=1,2,3,\;\sigma_i=\chi_{i+2},\;i=4,5,6,\;\sigma_7=\chi_4+\chi_8,\;\sigma_8=\chi_5+\chi_8,\;
\sigma_9=\chi_4+\chi_5+\chi_8.$$
Since there is no $\sigma_t$ which contains $\chi_4$ but don't contain $\chi_8$, it follows that $G$ is not a NAM-group. 
Moreover, if we let $\phi=\chi_4$ then 
$$0\leq \langle \phi, \sigma_t \rangle \leq \langle \chi_8,\sigma_t \rangle\text{ for all }1\leq t\leq 9.$$
Hence $G$ is not a BAM-group.
\end{exm}

\begin{exm}(See also \cite[Example 2.9]{cimrad})\label{ex3}
Let $G=A_6$, the alternating group of order $6$. The group $G$ has $7$ irreducible characters. 
Using GAP \cite{gap}, we deduce that $\Hilb(M(G))=\{\sigma_1,\ldots, \sigma_{16}\}$ where
\begin{align*}
& \sigma_1=(1,0,0,0,0,0,0),\; \sigma_2=(1,1,0,0,0,0,0),\; \sigma_3=(1,1,0,0,0,1,0),\; \sigma_4=(1,0,1,0,0,0,0), \\
& \sigma_5=(1,0,1,0,0,1,0),\; \sigma_6=(1,0,0,0,0,1,0),\; \sigma_7=(0,1,1,0,0,1,0),\; \sigma_8=(0,1,1,1,2,2,2), \\
& \sigma_9=(0,1,1,2,1,2,2),\; \sigma_{10}=(0,1,0,1,1,1,0),\; \sigma_{11}=(0,1,0,0,0,0,1),\; \sigma_{12}=(0,0,1,1,1,1,0), \\
& \sigma_{13}=(0,0,1,0,0,0,1),\; \sigma_{14}=(0,0,0,1,1,0,2),\; \sigma_{15}=(0,0,0,1,1,1,2),\; \sigma_{16}=(0,0,0,0,0,0,1). \\
\end{align*}
Note that for any pair $(i,j)\in \{1,2,\ldots,7\}\times\{1,2,\ldots,7\}$ with $(i,j)\notin \{(4,5),(5,4)\}$, there exists some $t=t(i,j)$
such that $\sigma_t(i)>0$ and $\sigma_t(j)=0$. On the other hand, there is no $\sigma_t$ with $\sigma_t(4)>0$ and $\sigma_t(5)=0$.
Thus $G$ is not a NAM-group. However, we have
$$ \sigma_{9}(4)=2>1=\sigma_{9}(5)\text{ and }\sigma_{8}(5)=2>1=\sigma_{8}(4).$$
Hence, $G$ is a WAM-group. Also, we claim that $G$ is a BAM-group. In order to show that, we fix some $1\leq k\leq 7$, we let
$\psi=(b_1,b_2,\ldots,b_7)$ with $b_i\in\mathbb Z$ and we assume that 
\begin{equation}\label{cucu}
0\leq \langle \psi, \sigma_t \rangle \leq \langle \chi_k,\sigma_t \rangle = \sigma_t(k)\text{ for all }1\leq t\leq 16.
\end{equation}
We have to prove that $\psi=0$ or $\psi=\chi_k$.

Let's assume that $k=1$. Applying \eqref{cucu} for $1\leq t\leq 1$ we deduce that
$$0\leq b_1\geq 1,\;b_2\leq 0,\;b_3\leq 0,\;b_2+b_6\leq 0,\;b_3+b_6\leq 0\text{ and }b_2+b_3=0.$$
Since $b_2,b_3\leq 0$ and $b_2+b_3=0$ it follows that $b_2=b_3=0$. Also, from above, we have $b_6\leq 0$.
Applying \eqref{cucu} for $t=16$ it follows that $b_7=0$. Applying \eqref{cucu} for $t=14$ and $t=15$ it follows
$$b_4+b_5=0\text{ and }b_4+b_5+b_6=0.$$
Hence $b_6=0$. From \eqref{cucu} applied for $t=9$ we get $b_4+2b_5=0$. Since $b_4+b_5=0$ it follows that $b_4=b_5=0$.
Thus $\psi=\chi_1$ or $\psi=0$, as required. For $2\leq k\leq 7$ the computations are similar, and we left them to the reader.
Hence, $G$ is indeed a BAM-group.
\end{exm}

\begin{teor}\label{t1}
Let $N \unlhd G$ be a normal subgroup of the WAM-group $G$. Then $G/N$ is a a WAM-group.
\end{teor}

\begin{proof}
Let $\tilde{\chi}_i$ and $\tilde{\chi}_{j}$ be two irreducible characters of $G/N$ and 
$\chi_i$ and $\chi_{j}$ their corresponding irreducible characters of $G$. Since $G$ is a WAM-group,
there exists a subgroup $H$ of $G$ and a linear character $\lambda$ of $H$ such that
$\langle \lambda^G,\chi_i \rangle > \langle \lambda^G,\chi_{j} \rangle \geq 0$. As in the proof of
\cite[Theorem 2.2.]{lucrare}, $\lambda$ induces a linear character $\tilde{\lambda}$ of $HN/N$. Moreover,
we have $\langle \tilde{\lambda}^{G/N},\tilde{\chi}_i \rangle = \langle \lambda^G,\chi_i \rangle$ and 
$\langle \tilde{\lambda}^{G/N},\tilde{\chi}_{j} \rangle = \langle \lambda^G,\chi_{j} \rangle$. In particular,
it follows that $$\langle \tilde{\lambda}^{G/N},\tilde{\chi}_i \rangle > \langle \tilde{\lambda}^{G/N},\tilde{\chi}_{j} \rangle.$$
Hence $G/N$ is weak almost monomial, as required.
\end{proof}

\begin{teor}\label{t2}
Let $G,G'$ be two finite groups. The following assertions are equivalent:
\begin{enumerate}
 \item[(1)] $G,G'$ are WAM-groups.
 \item[(2)] $G\times G'$ is a WAM-group.
\end{enumerate}
\end{teor}

\begin{proof}
$(1)\Rightarrow (2)$: As in the proof of \cite[Theorem 2.3.]{lucrare}, 
it follows from the fact that the irreducible characters of $G\times G'$ are of the form $\chi\times \chi'$, where 
$\chi\in\Irr(G)$ and $\chi'\in\Irr(G')$.

$(2)\Rightarrow (1)$: Follows from Theorem \ref{t1}.
\end{proof}

\section{Artin L-functions of WAM Galois groups}

Let $K/\mathbb Q$ be a finite Galois extension. For the character $\chi$ of a representation of the Galois group $G := Gal(K/\mathbb Q)$
on a finite dimensional complex vector space, let $L(s, \chi) := L(s, \chi, K/\mathbb Q)$ be the corresponding Artin L-function. 
Artin \cite{artin2} conjectured that $L(s, \chi )$ is holomorphic in $\mathbb C \setminus \{1\}$. Brauer \cite{brauer} proved that $L(s, \chi)$ is meromorphic
in $\mathbb C$.

Now, assume $\Irr(G)=\{\chi_1,\ldots,\chi_r\}$ and let $f_i=L(x,\chi_i)$, $1\leq i\leq r$, be the corresponding Artin L-functions.
For any virtual character $\phi=a_1\chi_1+\cdots+a_r\chi_r$ we have that
$$L(s, \phi)=f_1^{a_1}\cdots f_r^{a_r}.$$
Let $\Ar=\{L(s,\phi)\;:\;\phi\in \Ch_+(G)\}$ be the monoid of Artin $L$-functions.

Let $s_0\in \mathbb C \setminus \{1\}$ be fixed. We recall the fact that if $\lambda$ is a linear character of a subgroup $H$ of $G$ then
$$L(s, \lambda^G, K/\mathbb Q)=L(s, \lambda, K/F),$$
where $F=K^H$ is the fixed field of $H$. On the other hand, $L(s, \lambda, K/F)$ is a Hecke L-function and thus holomorphic at $s_0$.
It follows that $L(s,\phi)$ is holomorphic at $s_0$, for any monomial character $\sigma$.

Let $\Hol(s_0)$ be the submonoid of $\Ar$ consisting in L-functions which are holomorphic at $s_0$. Nicolae \cite{numb} proved
that $\Hol(s_0)$ is an affine monoid. Artin conjecture at $s_0$ can be stated as $$\Hol(s_0)=\Ar.$$
The following result extends \cite[Theorem 1]{monat}:

\begin{teor}\label{t3}
If the Galois group $G$ is weak almost monomial, then the following assertions are equivalent:
\begin{enumerate}
\item[(1)] Artin conjecture is true at $s_0$, i.e. $\Hol(s_0)=\Ar$.
\item[(2)] The monoid $\Hol(s_0)$ is factorial.
\end{enumerate}
\end{teor}

\begin{proof}
$(1)\Rightarrow (2)$: It is obvious. 

$(2)\Rightarrow (1)$: Suppose that Artin's conjecture is not true. Then there exists $1\leq k\leq r$ such that $\ord_{s=s_0}(f_k)<0$.
As in the proof of \cite[Theorem 1]{monat}, it follows that there is some $\ell\neq k$ such that $\ord_{s=s_0}(f_{\ell})>0$ and some
nonnegative integers $m_1,\ldots,m_r$ with $m_k>0$ such that the Hilbert basis of $\Hol(s_0)$ is 
$$\mathcal H = \{f_{\ell}^{m_1}f_1,\ldots,f_{\ell}^{m_r}f_r\}.$$
Since the Galois group $G$ is weak almost monomial
there exist a subgroup $H$ of $G$ and a linear character $\lambda$ of $H$ such that
$$a_k=\langle \sigma,\chi_k \rangle > \langle \sigma,\chi_{\ell} \rangle=a_{\ell}\geq 0,$$
where $\sigma=\lambda^G=a_1\chi_1+\cdots+a_r\chi_r$. In particular, we have
$$L(s,\sigma)=f_1^{a_1}\cdots f_r^{a_r}\in \Hol(s_0).$$
On the other hand, since $\mathcal H$ is the Hilbert basis of $\Hol(s_0)$, it follows that $L(s,\sigma)$ is a product
of elements from $\mathcal H$ and, as $m_k\geq 1$, we have $a_{k}\leq a_{\ell}$, a contradiction.
\end{proof}

The following result is the counterpart of \cite[Theorem 2.1]{lucrare} in the framework of weak almost monomial groups:

\begin{teor}\label{t4}
If the Galois group $G$ is weak almost monomial and there exists some $k$ such that $\ord_{s=s_0}(f_k)=1$ and $f_{\ell}(s_0)\neq 0$ for all $\ell\neq k$,
then all Artin L-functions of $K/\mathbb Q$ are holomorphic at $s_0$.
\end{teor}

\begin{proof}
Assume by contradiction that $s_0$ is a pole of some $L$-function $f_m$. Since $G$ is weak almost monomial, there exists a monomial character
$\sigma=a_1\chi_1+\cdots+a_r\chi_r$ such that $$a_m=\langle \sigma,\chi_m \rangle > \langle \sigma,\chi_k \rangle = a_k \geq 0.$$
Since the $L$-function $L(s,\sigma)=f_1^{a_1}\cdots f_r^{a_r}$ is holomorphic at $s_0$ and 
$$\ord_{s=s_0}(f_m^{a_m}) + \ord_{s=s_0}(f_k^{a_k}) \leq -a_m + a_k < 0, $$
it follows that there exists some $\ell\neq k$ such that $f_{\ell}(s_0)=0$ (and $a_{\ell}>0$), a contradiction.
\end{proof}

\section{GAP functions and computer experiments}

The following function in GAP \cite{gap} determines if a group is weak almost monomial.\\

\noindent
gap$>$ IsWAM:=function(g)\\
local cc,i,M,x,y,z,k,j,o,l,num;\\
if IsMonomial(g) then return(true);fi;\\
cc:=ConjugacyClassesSubgroups(g);\\
i:=Size(Irr(g));\\
M:=IdentityMat(i);\\
num:=i*(i-1);\\
for x in cc do\\
  y:=Representative(x); \\
  for z in LinearCharacters(y) do\\
    o:=InducedClassFunction(z,g);\\
    for j in [1..i] do for k in [1..i] do \\
      if M[j][k]=0 then\\
      if ScalarProduct(o,Irr(g)[j]) $>$ ScalarProduct(o,Irr(g)[k]) then \\
        M[j][k]:=1;num:=num-1;\\
        if num=0 then return(true);fi;\\
      fi;fi;
    od;od;
  od;
od;\\
return(false);\\
end;;\\

The following function in GAP \cite{gap} returns the Hilbert basis of the monoid generated by the
monomial characters of a finite group $G$:\\

\noindent
gap$>$LoadPackage("NumericalSgps");\\
HilbertBasisMonomial:=function(g) \\
local cc,i,x,y,z,o,l,j,t,a,s;\\
cc:=ConjugacyClassesSubgroups(g);\\
i:=Size(Irr(g));
t:=[];\\
for x in cc do\\
  y:=Representative(x); \\
  for z in LinearCharacters(y) do\\
    o:=InducedClassFunction(z,g);
    l:=[];\\
    for j in [1..i] do \\
      a:=ScalarProduct(Irr(g)[j],o);Add(l,a);
    od;\\
    Add(t,l);\\
  od; od;\\
s:=AffineSemigroup(t);\\
return(MinimalGenerators(s));\\
end;;\\

\begin{itemize}
\item 
The alternating groups $A_5$ and $A_9$ are NAM-groups; see \cite[Section 6]{cim} for a function in
GAP \cite{gap} which determines if a finite group is a NAM-group. Also, $A_5$ is a BAM-group; see \cite[Proposition 2.3]{book}.
The group $A_6$ is a BAM-group, hence weak almost monomial, but it is not a NAM-group; see Example \ref{ex3}. 
The groups $A_7$, $A_8$, $A_{10}$, $A_{11}$, $A_{12}$ and $A_{13}$ are not weak almost monomial. We believe that $A_n$ is not weak almost monomial for any 
$n\geq 7$ with $n\neq 9$.

\item The Higman-Sims group HS is not weak almost monomial.

\item The Hall-Janko groups J1 and J2 are not weak almost monomial.

\item The Mathieu group $M_{10}$ is weak almost monomial, but it is not a NAM-group.
The Mathieu groups $M_{11}$, $M_{12}$, $M_{21}$, $M_{22}$, $M_{23}$ and $M_{24}$ are not weak almost monomial.

\item The groups $SL_2(\mathbb F_5)$, $SL_2(\mathbb F_7)$, $SL_2(\mathbb F_9)$, $SL_2(\mathbb F_{11})$, $SL_2(\mathbb F_{13})$, $SL_2(\mathbb F_{17})$,
      $SL_2(\mathbb F_{19})$, $SL_2(\mathbb F_{23})$, $SL_2(\mathbb F_{25})$ are not weak almost monomial. Also, the groups			
			$SL_3(\mathbb F_2)$, $SL_3(\mathbb F_3)$ and $SL_2(\mathbb F_5)$ are not weak almost monomial.
\end{itemize}

{}

\end{document}